\documentclass[11pt]{article}
\usepackage{geometry}                		

\usepackage{amssymb}
\usepackage{amsmath}

\usepackage{palatino}

\usepackage{graphicx}
\usepackage[english]{babel}
\usepackage{amscd}
\usepackage{amsthm}

\usepackage{color}
\usepackage{enumerate}

\newtheorem{theorem}{Theorem}[section]
\newtheorem{lemma}[theorem]{Lemma}

\newtheorem{conjecture}[theorem]{Conjecture}

\newtheorem{problem}[theorem]{Problem}

\newcommand{\F}{\mathcal{F}}
\newcommand{\R}{\mathbb R}
\newcommand{\N}{\mathbb N}
\newcommand{\Ha}{\mathcal{H}}
\newcommand{\Le}{\mathcal{L}}

\newcommand{\bx}{\mathbf{x}}
\newcommand{\by}{\mathbf{y}}
\newcommand{\inter}{\operatorname{int}}

\title{On $k$-antichains in the unit $n$-cube}

\author{
Christos Pelekis
\thanks{
Institute of Mathematics, Czech Academy of Sciences, \v{Z}itn\'a 25, 
115 67,   Praha 1, 
Czech Republic.  
Research supported by the GA\v{C}R project 18-01472Y and RVO: 67985840. 
E-mail: pelekis.chr@gmail.com
}
\and
V\'aclav Vlas\'ak 
\thanks{
Faculty of Mathematics and Physics, Charles University, Sokolovsk\'a 83, 18675 Praha 8, Czech Republic. E-mail:  vlasakvv@gmail.com
} 
}

\begin{document}
	\maketitle
	
\begin{abstract}
A \emph{chain} in the unit $n$-cube is a set $C\subset [0,1]^n$ such that  for every 
$\mathbf{x}=(x_1,\ldots,x_n)$ and $\mathbf{y}=(y_1,\ldots,y_n)$ in  $C$ we either have  
$x_i\le y_i$ for all $i\in [n]$, or $x_i\ge y_i$ for all $i\in [n]$. We consider 
subsets, $A$, of the unit $n$-cube $[0,1]^n$ that satisfy 
\[ 
\text{card}(A \cap C) \le k, \, \text{ for all chains } \, C \subset [0,1]^n \, , 
\]
where $k$ is a fixed positive integer. 
We refer to such a set $A$ as a $k$-antichain. We show that the $(n-1)$-dimensional Hausdorff measure of a $k$-antichain in $[0,1]^n$ is at most $kn$ and that the bound is asymptotically sharp. 
Moreover, we conjecture that there exist $k$-antichains in $[0,1]^n$ whose 
$(n-1)$-dimensional Hausdorff measure equals $kn$ and we verify the validity of this conjecture when $n=2$.  
\end{abstract}



\section{Prologue, related work and main results}

Let $[n]$ denote the set of positive integers $\{1,\ldots,n\}$, and $2^{[n]}$ denote the collection  of all subsets of $[n]$.  
Given two points $\bx=(x_1,\ldots, x_n)$ and $\by=(y_1,\ldots,y_n)$ in $\R^n$, we write $\bx \le \by$ if $x_i\le y_i$, for all $i\in [n]$. Given a subset $S\subset \R^n$, we say that 
a set $C\subset S$ is a {\em chain in $S$} if for all $\bx,\by\in C$ it either holds $\bx\le \by$ or $\by \le \bx$. 
Given a non-negative real number $s$, we denote by $\Ha^s(\cdot)$ the $s$-dimensional Hausdorff outer measure (see~\cite[p.~81 and p.~1--2]{Evans_Gariepy}). Notice that 
$\Ha^0(\cdot)$ is counting measure. 
Finally, given a positive integer $k$ and  a set $S\subset \R^n$,  
a {\em $k$-antichain in $S$} is a set $A\subset S$ 
such that $\Ha^0(A\cap C) \le k$, for all chains $C\subset S$. An $1$-antichain is simply referred to as an {\em antichain}.   

This work is motivated by a particular result from {\em extremal set theory}. Extremal set theory  (see \cite{Anderson, Engeltwo}) is a rapidly growing branch of combinatorics  which is  concerned with the problem of obtaining sharp  estimates on the size of a collection $\F\subset 2^{[n]}$, subject to constraints that are expressed in terms of union, intersection or inclusion. A particular line of research 
is driven by the idea that several results from extremal combinatorics  
have continuous counterparts. This is an idea that goes back to the 70's (see~\cite{Nash-Williams}) and,  since its conception, has resulted in reporting several analogues of results from extremal combinatorics both  in a ``measure-theoretic context"  (see, for example, \cite{DMP, engel, EMP, EMPR, katonaone, katonatwo, klainrota, MPV}) as well as in a ``vector space context" (see, for example, \cite{blokhuis, FranklWilson, katchalski})
In this note we report yet another measure-theoretic analogue of a result from extremal combinatorics. 

Before being more precise, let us remark that one can associate a binary vector of length $n$ to every $F\subset [n]$: simply put $1$ in the $i$-th coordinate if $i\in F$, and $0$ otherwise. Notice that this correspondence is bijective, and one may choose to not distinguish between subsets of $[n]$ and elements of $\{0,1\}^n$. In other words, any statement regarding collections $\F\subset 2^{[n]}$ can be turned to a statement regarding subsets $F\subset \{0,1\}^n$, and vice versa.

Perhaps the most fundamental result in extremal set theory is due to Sperner~\cite{Sperner}. 
It provides a sharp upper bound on the cardinality 
of an antichain in $\{0,1\}^n$. Sperner's theorem is a well-known and celebrated result that has been generalised in a plethora of ways (see~\cite{Engeltwo} for a textbook devoted to the topic). A particular extension of Sperner's theorem is due to Paul Erd\H{o}s, and reads as follows. 

\begin{theorem}[Erd\H{o}s~\cite{Erdos}]
\label{erdos_thm}
Fix a positive integer $k\in [n]$. If $A$ is a $k$-antichain in $\{0,1\}^n$, then 
\[
\Ha^0(A)\le \sum_{i=1}^{k} \binom{n}{\lfloor \frac{n-k}{2}\rfloor +i}\, .
\]
\end{theorem}

Notice that the bound provided by Theorem~\ref{erdos_thm} is sharp and is attained by the set 
\[
A = \bigcup_{i=1}^k \left\{\bx =(x_1,\ldots,x_n) \in \{0,1\}^n: \sum_{i=1}^{n}x_i = \lfloor \frac{n-k}{2}\rfloor +i \right\}\, .
\] 
In other words, Erd\H{o}s' result provides a sharp upper bound on the size of a $k$-antichain 
in the binary $n$-cube $\{0,1\}^n$. 
In this article we investigate a continuous analogue of Theorem~\ref{erdos_thm}. 
There are several ways to consider Theorem~\ref{erdos_thm} in a continuous setting (see~\cite{MPV} for an alternative direction), but the main idea is to examine what happens when one replaces the binary $n$-cube $\{0,1\}^n$ with the unit $n$-cube $[0,1]^n$ in Theorem~\ref{erdos_thm}.
What is the maximum ``size" of a $k$-antichain in the unit $n$-cube $[0,1]^n$?
Since we are dealing with subsets of $[0,1]^n$ and we have to choose an adequate notion of ``size". A first choice could be the $n$-dimensional Lebesgue measure, denoted $\Le^n(\cdot)$. 
However, it  is not difficult to see, using Lebesgue's density theorem, that the $\Le^n$-measure of a $k$-antichain equals zero. Given this fact, it is therefore natural to ask for sharp upper bounds on the Hausdorff dimension and the corresponding Hausdorff measure of a $k$-antichain in the unit $n$-cube. In the case of antichains this has been considered in~\cite{EMPR}, where the following continuous analogue of Sperner's theorem has been reported.  

\begin{theorem}[Engel et al.~\cite{EMPR}]
\label{empr}
If $A$ is an antichain in $[0,1]^n$, then 
\[
\Ha^{n-1}(A)\le n \, .
\]
\end{theorem}

In particular, the Hausdorff dimension of an antichain is at most $n-1$. 
Let us remark that the bound provided by Theorem~\ref{empr} is asymptotically sharp. Indeed,  as is observed in~\cite{EMPR}, this can be seen by considering the boundary of $\ell^p$-unit balls, i.e., by considering the sets 
\[
A_p = \left\{ \bx \in [0,1]^n : \|\bx\|_p = 1 \right\} \, ,
\]
as $p\to\infty$. Notice that $A_p$ is an antichain in $[0,1]^n$, but $A_{\infty}$ is {\em not}. Moreover, notice that $\Ha^{n-1}(A_{\infty})=n$. 
Now, it is not difficult to see that the $p$-ball $B_p = \{\bx\in \R^n : \|\bx\|_p \le 1\}$ converges, with respect to the Hausdorff distance, to the $\infty$-ball $B_{\infty} = \{\bx\in \R^n : \|\bx\|_{\infty} \le 1\}$. 
Furthermore, it is known (see~\cite[p.~219]{Schneider}) that whenever a sequence of convex bodies $B_i$ converges, with respect to the Hausdorff distance, to a convex body $B$, then it follows that 
$\Ha^{n-1}(\partial B_i)$ converges to $\Ha^{n-1}(B)$. Hence $\Ha^{n-1}(A_p)$ tends to $n$, as $p\to\infty$, and therefore one can find an antichain in $[0,1]^n$ whose $\Ha^{n-1}$-measure is arbitrarily  close to $n$. 
There remains the question of whether there exists an antichain whose $\Ha^{n-1}$-measure is equal to $n$. The following conjecture has been put forward in~\cite{EMPR}. 

\begin{conjecture}[Engel et al.~\cite{EMPR}]
\label{conj_empr}
There exists an antichain in $[0,1]^n$ such that $\Ha^{n-1}(A) = n$. 
\end{conjecture}

When $n=1$ this conjecture is clearly true, and when $n=2$ it is observed in~\cite{EMPR} that the validity of Conjecture~\ref{conj_empr} is an immediate consequence of the following, well-known, result. 
Recall that a {\em singular function} $f:[a,b]\to [c,d]$ is a strictly decreasing function whose derivative equals zero almost everywhere. 

\begin{theorem}[Folklore] 
\label{folk}
Let $f:[a,b]\to [c,d]$ be a singular function and let $G_f =\{(x,f(x)): x\in [a,b]\}$ be its graph. Then 
$\Ha^1(G_f) = (b-a) + (d-c)$. 
\end{theorem}

We refer the reader to~\cite[p.~101]{Saks} for details regarding the existence of singular functions, and to~\cite[p.~810]{Foran} for a sketch of a proof of Theorem~\ref{folk}. Since the graph of a singular function $f:[0,1]\to [0,1]$ is clearly an antichain in $[0,1]^2$, it follows that Conjecture~\ref{conj_empr} holds true when $n=2$. 

In this note we focus on $k$-antichains in $[0,1]^n$, for $k>1$. Using Theorem~\ref{empr}, we obtain the following upper bound on the maximum ``size" of a $k$-antichain in the unit $n$-cube. 

\begin{theorem}\label{k_ant_thm} 
Fix a positive integer $k\ge 1$. If $A$ is a $k$-antichain in $[0,1]^n$, then 
\[
\Ha^{n-1}(A)\le kn \, . 
\]
\end{theorem}

Using a similar argument as the one used in the remarks after Theorem~\ref{empr}, it can be shown that the upper bound provided by Theorem~\ref{k_ant_thm} is asymptotically sharp, and it is therefore natural to ask whether there exist $k$-antichains in $[0,1]^n$ whose $\Ha^{n-1}$-measure is equal to $kn$. 
We conjecture that the answer is in the affirmative, for all $n\ge 2$, and in this note we 
verify the validity of this conjecture for $n=2$.

\begin{theorem}\label{sharpness} 
There exists a $k$-antichain in $[0,1]^2$ such that $\Ha^1(A)=2k$. 
\end{theorem}

\section{Proofs}

\begin{proof}[Proof of Theorem~\ref{k_ant_thm}]
It is enough to show that there exist $k$ sets $A_1,\ldots,A_k\subset [0,1]^n$ such that  $A=\cup_{i=1}^{k}A_i$ and each $A_i$ is an antichain. Theorem~\ref{k_ant_thm} then follows from 
Theorem~\ref{empr}. We prove the required result by induction on $k$. The case $k=1$ is clear. 
Assuming that the result holds true for $k-1>1$, we prove it for $k$. 
Let $B$ be the set consisting of all minimal elements of $A$. That is, let 
\[
B = \{\bx\in A : \text{ there is no } \by\in A\setminus \{x\} \text{ satisfying } \by \le \bx \} \, .
\]
Clearly, $B$ is an antichain and it is enough to show that $A\setminus B$ is a $(k-1)$-antichain in $[0,1]^n$; the result then follows from the induction hypothesis. 
Assume, towards a contradiction, that $A\setminus B$ is not a  $(k-1)$-antichain. This implies that 
there exists a chain $C\subset [0,1]^n$ such that $\Ha^0((A\setminus B) \cap C) \ge k$.  
Let $\by\in (A\setminus B) \cap C$ be a minimal element, i.e, $\by$ is such that there does not exist 
$\mathbf{z}$, which is distinct from $\by$, satisfying  
$\mathbf{z}\in (A\setminus B) \cap C$ and $\mathbf{z}\le \by$.  Notice that the existence of $\by$ 
follows from the fact that, since $A$ is a $k$-antichain,  $(A\setminus B) \cap C$ is a finite set. 
Since $\by\notin B$ it follows that there exists $\bx\in A$ such that $\bx\neq\by$ and $\bx\le \by$. 
Now set $D:= \{\bx\} \cup (A\setminus B) \cap C$ and notice that $D$ is a chain that satisfies 
$\Ha^0(D\cap A)\ge k+1$, contrariwise to the fact that $A$ is a $k$-antichain. The result follows. 
\end{proof}

We proceed with the proof of Theorem~\ref{sharpness}. This requires some additional piece of notation. 
Given two functions $g,h:[0,1]\to [0,1]$, let 
\[
W(g,h) := \{  (x,y)\in [0,1]^2 : g(x) \le y \le h(x) \}\, .
\] 
Given a function $g:[0,1]\to [0,1]$, let $Gr(g) =\{(x,y)\in [0,1]^2: y = g(x)\}$ be its graph. 
If $A\subset [0,1]^2$, we denote its interior by $\inter(A)$. 
Finally, given two points $\bx=(x_1,x_2), \by=(y_1,y_2)\in \R^2$ with $x_1<y_1$ and 
$x_2 > y_2$, let 
\[
R[\bx,\by] := \{ (z_1,z_2)\in \R^2 : z_1\in [x_1,y_1] \text{ and } z_2\in [y_2, x_2]  \}
\]
be the rectangle ``determined" by the points $\bx,\by$. The proof of Theorem~\ref{sharpness} relies upon the following. 

\begin{lemma}
\label{help}
Let $g,h:[0,1]\to [0,1]$ be strictly decreasing and continuous bijections such that $g(x) < h(x)$, for all $x\in (0,1)$. Then there exists a strictly decreasing function $D:(0,1)\to (0,1)$ such that
\begin{itemize}

\item[(a)] $g(x)\leq D(x)\leq h(x)$ for every $x\in(0,1)$,
\item[(b)] $\Ha^1(Gr(D))=2$. 

\end{itemize}
\end{lemma}
\begin{proof}
Consider the function $f:[0,1]\to[0,1]$ defined by 
\[
f(x):=\frac12(g(x)+h(x)), \text{ for } x\in[0,1]\, .
\]
Clearly, $f$ is a strictly decreasing, continuous, bijection and $g(x)<f(x)<h(x)$ holds true for every $x\in(0,1)$. We will show that we can inductively construct sequences $\{x_n\}_n$ and $\{y_n\}_n$ that satisfy the following five conditions:
\begin{itemize}

\item[(i)] $\frac12=x_1>x_2>\dots>0$ and  $\frac12=y_1<y_2<\dots<1$,
\item[(ii)] $R[(x_{n+1},f(x_{n+1})),(x_n,f(x_n))]\subset W(g,h)$, $n\in\N$, 
\item[(iii)] $R[(y_n,f(y_n)),(y_{n+1},f(y_{n+1}))]\subset W(g,h)$, $n\in\N$,
\item[(iv)] $R[(x_{n+1},f(x_{n+1})),(x_n,f(x_n))]\not\subset \inter(W(g,h))$, $n\in\N$, 
\item[(v)] $R[(y_n,f(y_n)),(y_{n+1},f(y_{n+1}))]\not\subset \inter(W(g,h))$, $n\in\N$.

\end{itemize}
We first show how to construct the sequence $\{x_n\}_n$. 
Begin by setting $x_1=\frac12$. 
Now, assuming we have already constructed $x_1,\dots,x_n$ satisfying (i), (ii) and (iv), we show how to construct $x_{n+1}$. 
By (i) we have $1>x_n>0$. Since $g,f,h$ are strictly decreasing functions and $g(x)<f(x)<h(x)$ holds true,  for every $x\in(0,1)$, it follows that 
\[
0<g^{-1}(f(x_n))<x_n \, \text{ and } \, 0<f^{-1}(h(x_n))<x_n \, . 
\]
Now set $x_{n+1}:=\max\{g^{-1}(f(x_n)),f^{-1}(h(x_n))\}$. Clearly, it holds  $0<x_{n+1}<x_n$ as well as  $R[(x_{n+1},f(x_{n+1})),(x_n,f(x_n))]\subset W(g,h)$ and $R[(x_{n+1},f(x_{n+1})),(x_n,f(x_n))]\not\subset \inter(W(g,h))$. So $x_1,\dots,x_{n+1}$ satisfy (i), (ii) and (iv). Thus we finished the construction of the sequence $\{x_n\}$. The sequence $\{y_n\}_n$ can be constructed similarly; we leave the details to the reader. 

Since the sequences $\{x_n\}$ and $\{y_n\}$ are monotone and bounded, there exists the limits 
\begin{equation}
\label{lim}
x:=\lim_{n\to\infty}x_n \, \text{ and } \, y:=\lim_{n\to\infty}y_n.
\end{equation}
We now show that $x=0$. Assume, towards a contradiction, that $x\neq 0$. Clearly, it holds  
\begin{equation}
\label{cont}
0<x<x_n, \, \text{ for every } \,  n \in \mathbb{N} \, .
\end{equation}
Since $(x,f(x))\in\inter(W(g,h))$, there exists $\delta>0$ such that for every $y,z\in(x-\delta,x+\delta)$ satisfying $y<z$ we have $R[(y,f(y)),(z,f(z))]\subset\inter(W(g,h))$. By (\ref{lim}) it follows that there exists  $n\in\N$ such that $x_{n-1}\in(x-\delta,x+\delta)$. Then (iv) implies that $x_n\leq x-\delta<x$ which contradicts (\ref{cont}). Hence it holds $x=0$. In a similar way, it can be shown that $y=1$. 

Since $f$ is continuous we have 
$$\lim_{n\to\infty}(x_n,f(x_n))=(0,f(0))=(0,1),$$
$$\lim_{n\to\infty}(y_n,f(y_n))=(1,f(1))=(1,0).$$
Since $x_1=y_1$ it follows that 
\[
\sum_{n=1}^{\infty}(x_n-x_{n+1})=\frac12 \quad \text{ and } \quad  \sum_{n=1}^{\infty}(y_{n+1}-y_n)=\frac12
\]
as well as 
\[
\sum_{n=1}^{\infty}(f(x_{n+1})-f(x_n))=1-f\left(1/2\right)  \quad \text{ and } \quad
\sum_{n=1}^{\infty}(f(y_n)-f(y_{n+1}))=f\left(1/2\right)
\]
and therefore we conclude 
\begin{equation}\label{sum}
\sum_{n=1}^{\infty}\bigg((x_n-x_{n+1})+(y_{n+1}-y_n)+(f(x_{n+1})-f(x_n))+(f(y_n)-f(y_{n+1}))\bigg)=2 \, .
\end{equation}
Now Theorem~\ref{folk} implies that for every $n\in\N$ there exist strictly decreasing functions $d_{x,n},d_{y,n}$  that satisfy the following four conditions:
\begin{itemize}

\item[(A)] $d_{x,n}:[x_{n+1},x_n]\to[f(x_n),f(x_{n+1})]$,
\item[(B)] $d_{y,n}:[y_n,y_{n+1}]\to[f(y_{n+1}),f(y_n)]$,
\item[(C)] $\Ha^1(Gr(d_{x,n}))=(x_n-x_{n+1})+(f(x_{n+1})-f(x_n))$,
\item[(D)] $\Ha^1(Gr(d_{y,n}))=(y_{n+1}-y_n)+(f(y_n)-f(y_{n+1}))$.

\end{itemize}
Gluing  those functions together, we obtain desired function $D:(0,1)\to(0,1)$. Indeed, by (A), (B), (ii) and (iii) we have
\begin{equation}\nonumber
\begin{aligned}
Gr(D)&:=\bigcup_{n=1}^{\infty}(Gr(d_{x,n})\cup Gr(d_{y,n}))\\
&\subset\bigcup_{n=1}^{\infty}\bigg(R[(x_{n+1},f(x_{n+1})),(x_n,f(x_n))]\cup R[(y_n,f(y_n)),(y_{n+1},f(y_{n+1}))]\bigg)\\
&\subset W(g,h)
\end{aligned}
\end{equation}
and so $D$ satisfies (a). Using (\ref{sum}), (C) and (D) we conclude that
\[
\Ha^1(Gr(D))=\sum_{n=1}^{\infty}(\Ha^1(Gr(d_{x,n}))+\Ha^1(Gr(d_{y,n})))=2 \,  
\]
and therefore $D$ also satisfies (b). The lemma follows. 
\end{proof}

We are now ready to prove Theorem~\ref{sharpness}. 

\begin{proof}[Proof of Theorem~\ref{sharpness}]

Clearly, there exist continuous and strictly decreasing bijections $f_i:[0,1]\to[0,1]$, $i\in[2k]$, such that 
\begin{equation}\label{usp}
f_1(x)>f_2(x)>\dots>f_{2k}(x), \, \text{ for every } \, x\in (0,1) \, .
\end{equation}
By Lemma~\ref{help} we can find for every $i\in[k]$ strictly decreasing functions $D_i:(0,1)\to(0,1)$ such that
\begin{itemize}

\item[($\alpha$)] $f_{2i}(x)\leq D_i(x)\leq f_{2i-1}(x)$ for every $x\in(0,1)$,
\item[($\beta$)] $\Ha^1(Gr(D_i))=2$.

\end{itemize}
Now consider the set $A:=\bigcup_{i=1}^kGr(D_i)$. Since $D_i$ is a strictly decreasing function, it follows that $Gr(D_i)$ is an antichain for every $i\in[k]$, and therefore $A$ is $k$-antichain. Since $Gr(D_i)\subset(0,1)^2$ for every $i\in[k]$, we have $A\subset(0,1)^2$. By ($\alpha$) and (\ref{usp}) we have $Gr(D_i)\cap Gr(D_j)=\emptyset$ for every $i,j\in[k]$, $i\neq j$. Thus, by ($\beta$) we have 
\[
\Ha^1(A)=\sum_{i=1}^k\Ha^1(Gr(D_i)=2k \, ,
\]
as desired. 
\end{proof}

\section{Concluding remarks}

As mentioned in the introduction, there are several ways to consider Theorem~\ref{erdos_thm} in a continuous setting, and an alternative direction has been considered in~\cite{MPV}. It is shown in~\cite{MPV} that given $s\in[0,1]$ and  $\beta\ge 0$ there exists a set $A\subset [0,1]^n$ that satisfies 
$\dim_H(A)=n-1+s$ and $\Ha^s(A\cap C) \le \beta$, for all chains $C\subset [0,1]^n$. Here, $\dim_H(\cdot)$ denotes Hausdorff dimension (see~\cite[p.~86]{Evans_Gariepy}). Given this result, 
the following problem arises naturally. 

\begin{problem}[Mitsis et al.~\cite{MPV}]
\label{prbl_mpv}
Fix $s\in [0,1]$ and $\beta\ge 0$.  
Let  $A\subset [0,1]^n$ be a measurable set such that $\dim_H(A) = n-1 + s$ and 
$\Ha^s(A\cap C) \le \beta$,  for all chains $C\subset [0,1]^n$. 
What is a sharp upper bound on $\Ha^{n-1+s}(A)$?
\end{problem}

The case $s=0, \beta=1$ has been considered in~\cite{EMPR}. The case $s=0, \beta\in\N$ has been the content of the present article. The case $s=1, \beta\in (0,n]$ has been considered in~\cite{MPV}. The problem remains open for all other values of the parameters $s,\beta$.

\end{document}